\documentclass[12pt]{article}
\usepackage{amssymb,amsmath,amsthm,psfrag,graphics,latexsym,graphicx}

\numberwithin{equation}{section}
\newtheorem{theorem}{Theorem}[section]
\newtheorem{lemma}[theorem]{Lemma}

\newtheorem{remark}[theorem]{Remark}
\newtheorem{question}[theorem]{Question}
\newtheorem{construction}[theorem]{Construction}

\title{Bipartite graphs whose squares are not chromatic-choosable}

\author{Seog-Jin Kim \\ 
\small Department of Mathematics Education\\[-0.8ex]
\small Konkuk University\\[-0.8ex]
\small Seoul, 143-701, Korea \\
\small\tt skim12@konkuk.ac.kr \\
\and
Boram Park\thanks{Corresponding author: borampark@nims.re.kr} \\
\small Division of Mathematical Modelling \\  [-0.8ex]
\small National Institute for Mathematical Sciences \\[-0.8ex]
\small Daejeon, 305-811, Korea \\
\small\tt borampark@nims.re.kr
}

\begin{document}

\maketitle

\begin{abstract}
The square $G^2$ of  a graph $G$ is the graph defined on $V(G)$ such that two vertices $u$ and $v$ are adjacent in $G^2$ if the distance between $u$ and $v$ in $G$ is at most 2.  Let $\chi(H)$ and $\chi_{\ell}(H)$ be the chromatic number and the list chromatic number of $H$, respectively.  A graph $H$ is called {\em chromatic-choosable} if $\chi_{\ell} (H) = \chi(H)$.  It is an interesting problem to find graphs that are
chromatic-choosable.

Motivated by the List Total Coloring Conjecture, Kostochka and Woodall (2001) proposed the List Square Coloring Conjecture which states that $G^2$ is chromatic-choosable for every graph $G$. 
Recently, Kim and Park showed that the List Square Coloring Conjecture does not hold in general
by finding a family of graphs whose squares are complete multipartite graphs with partite sets of unbounded size. 
It is a well-known fact that the List Total Coloring Conjecture is true if the List Square Coloring Conjecture holds for special class of bipartite graphs.
On the other hand, the counterexamples to the List Square Coloring Conjecture are not bipartite graphs. Hence a natural question is whether $G^2$ is chromatic-choosable or not for every bipartite graph $G$.

In this paper, we give a bipartite graph $G$ such that $\chi_{\ell} (G^2) \neq \chi(G^2)$.  Moreover, we show  that the value $\chi_{\ell}(G^2) - \chi(G^2)$ can be arbitrarily large.
\end{abstract}


\noindent
{\bf Keywords:} Square of graph, chromatic-choosable, list chromatic number

\noindent
{\bf 2010 Mathematics Subject Classification: 05C15}

\section{Introduction}

A proper $k$-coloring $\phi: V(G) \rightarrow \{1, 2, \ldots, k \}$ of a graph $G$ is an assignment of colors to the vertices of $G$ so that any two adjacent vertices  receive distinct colors.
The {\em chromatic number} $\chi(G)$ of $G$ is the least $k$ such that there exists a proper $k$-coloring of $G$.

A {\em list assignment} on $G$ is a function
$L$ that assigns each vertex $v$ a set $N(v)$ which is
a list of available colors at $v$.
A graph $G$ is said to be {\em $k$-choosable} if for any list assignment $L$ such that
$|L(v)| \geq k$ for every vertex $v$, there exists a proper coloring $\phi$ such that $\phi(v) \in L(v)$ for every $v \in V(G)$.
The least $k$ such that $G$ is $k$-choosable is called the {\it list chromatic number} $\chi_\ell(G)$ of  $G$.
Clearly $\chi_{\ell}(G) \geq \chi(G)$ for every graph $G$.

A graph $G$ is called {\em chromatic-choosable} if $\chi_{\ell} (G) = \chi(G)$. It is an interesting problem to determine which graphs are chromatic-choosable.  There are several famous conjectures that some  classes of graphs are chromatic-choosable including the List Coloring Conjecture (see \cite{Toft} for detail).

Motivated by the List Total Coloring Conjecture, Kostochka and Woodall \cite{KW2001} proposed
the List Square Coloring Conjecture which states that $G^2$ is chromatic-choosable for every graph $G$. It was noted in \cite{KW2001} that the List Total Coloring Conjecture is true if the List Square Coloring Conjecture is true.
The List Square Coloring Conjecture has attracted a lot of  attention and been cited in many papers related with coloring problems.
Recently, Kim and Park \cite{KP2014-LSCC}
disproved the List Square Coloring Conjecture
by finding a family of graphs whose squares are complete multipartite graphs with partite sets of unbounded size.  Later, two different types of counterexamples to the List Square Coloring Conjecture have been known in \cite{KKB, Kosar}.   

If $H$ is the graph obtained by placing a vertex in the middle of every edge of a graph $G$, then
$H^2 = T(G)$, where $T(G)$ is the total graph of $G$.  Hence if the List Square Coloring Conjecture is true for a special class of bipartite graphs, then the List Total Coloring Conjecture is true. (see  \cite{KW2001} for detail.)

On the other hand, all of the counterexamples in \cite{KP2014-LSCC, KKB, Kosar} to the List Square Coloring Conjecture  are not bipartite graphs.
Hence a natural interesting question is whether $G^2$ is chromatic-choosable when $G$ is a bipartite graph.
This question was raised
by an anonymous referee and appeared in \cite{KP2014-LSCC}.
In this paper, we will give a bipartite graph $G$ such that $\chi_{\ell}(G^2) \neq \chi(G^2)$.
Moreover, we show that the gap between $\chi_{\ell} (G^2)$  and $\chi(G^2)$ can be arbitrarily large.
\section{Construction}
Let $[n]$ denote $\{1,2,\ldots,n\}$.
A {\em Latin square of order} $n$ is an $n \times n$
array such that every cell contains an element of $[n]$ and every element of $[n]$ occurs exactly
once in each row and each column.
For a Latin square $L$ of order $n$, the element on the $i$th row and the $j$th column is denoted by $L(i,j)$.
Two Latin squares $L_1$ and $L_2$ are \textit{orthogonal} if for any $(i,j) \in [n]\times [n]$, there exists unique $(k,\ell)\in[n]\times [n]$ such that $L_1(k,\ell)=i$ and $L_2(k,\ell)=j$.

From now on, we fix a prime number $n$ with $n\ge 3$.
For $i\in [n-1]$,  we define a Latin square $L_i$ of order $n$ by
\begin{eqnarray} \label{eq:Latin}
 L_i(j,k)= j+i(k-1) \pmod{n}, \quad \mbox{ for } (j, k) \in [n] \times [n].
\end{eqnarray}
Then it is easily checked  (and well-known) that $L_i$ is a Latin square of order $n$ and $\{L_1,L_2,\ldots, L_{n-1}\}$ is a  family of mutually orthogonal Latin squares of order $n$ as $n$ is prime (see page 252 in \cite{Van-Lint}).  In Figure \ref{Latin-square}, $L_1$ and  $L_2$ are orthogonal  Latin squares defined in (\ref{eq:Latin}) for $n = 3$.

\begin{figure}[b!]
\vspace{0.5cm}
\[
L_1={\footnotesize\begin{tabular}{|c|c|c|}
                  \hline
                  1 & 2 & 3 \\ \hline
                  2 & 3 & 1 \\ \hline
                  3 & 1 & 2 \\
                  \hline
                \end{tabular}} \hspace{1cm}
                L_2={\footnotesize\begin{tabular}{|c|c|c|}
                  \hline
                  1 & 3 & 2 \\ \hline
                  2 & 1 & 3 \\ \hline
                  3 & 2 & 1 \\
                  \hline
                \end{tabular}}
\]
\caption{Latin squares $L_1$ and $L_2$ of order $3$ defined in (\ref{eq:Latin}).}
\label{Latin-square}
\end{figure}

\medskip

Now we will construct a bipartite graph $G$ such that $G^2$ is not chromatic-choosable.  First, we will describe briefly how to construct such bipartite graph $G$, and then will give a formal description in Construction \ref{construction}.

\bigskip

\noindent {\bf The procedure of the construction of $G$} \\

\noindent {\bf Step 1: }  For each prime number $n \geq 3$, we construct a graph $H_n$ with $2n^2$ vertices as follows.
For $k\in [n]$, let $P_k$ be the set of $n$ elements such that
\begin{eqnarray*}
P_k&=&\{ v_{k,1}, v_{k,2}, ..., v_{k,n} \},
\end{eqnarray*}
and for $i\in [n-1]$, let $Q_i$ be the set of $n$ elements such that
\begin{eqnarray*}
Q_i&=&\{ w_{i,1}, w_{i,2}, ...., w_{i,n}\},
\end{eqnarray*}
and let $S$ be the set of $n$ elements such that
\begin{eqnarray*}
S&=&\{ s_{1}, s_{2}, ...., s_{n}\}.
\end{eqnarray*}
Let $\{L_1,L_2,\ldots, L_{n-1}\}$ be the family of mutually orthogonal Latin squares of order $n$
obtained by (\ref{eq:Latin}).  Graph $H_n$ is defined as follows:
\begin{eqnarray*}
V(H_n) &=& \left(  \cup_{k =1}^{n} P_k  \right) \bigcup \ \left( \cup_{i =1}^{n-1} Q_i \right) \bigcup S,\\
E(H_n) &=& \left( \bigcup_{i\in [n-1]}\bigcup_{j\in [n]} \{ w_{i,j}  v_{k,L_i(j,k)} : k\in  [n] \} \right) \bigcup
\left( \bigcup_{j\in [n]} \{ s_j v_{k,j} : k \in  [n] \} \right).
\end{eqnarray*}
Let $T_j=\{ v_{1,j}, v_{2,j}, \ldots, v_{n,j}\}$ for each $j\in [n]$.
Note that  for each vertex $w_{i,j}$, $N_{H_n} (w_{i, j}) = \{v_{k,L_i(j,k)} :  k \in[ n] \}$, and
$H_n$ is the graph obtained by removing the edges in $\bigcup_{j \in [n]} \{ xy : x,y\in T_j\}$ from the graph $G_n$ in \cite{KP2014-LSCC} and adding vertices of $S$ and the edges of $\bigcup_{j \in [n]} \{ xs_j : x\in T_j\}$. 
(Figure \ref{fig1} is the case when $n=3$.)

\begin{figure}[t!]
\begin{center}
\psfrag{a}{\tiny$v_{1,1}$}
\psfrag{b}{\tiny$v_{1,2}$}
\psfrag{c}{\tiny$v_{1,3}$}
\psfrag{d}{\tiny$v_{2,1}$}
\psfrag{e}{\tiny$v_{2,2}$}
\psfrag{f}{\tiny$v_{2,3}$}
\psfrag{g}{\tiny$v_{3,1}$}
\psfrag{h}{\tiny$v_{3,2}$}
\psfrag{i}{\tiny$v_{3,3}$}
\psfrag{m}{\tiny$w_{1,1}$}
\psfrag{n}{\tiny$w_{1,2}$}
\psfrag{o}{\tiny$w_{1,3}$}
\psfrag{p}{\tiny$w_{2,1}$}
\psfrag{q}{\tiny$w_{2,2}$}
\psfrag{r}{\tiny$w_{2,3}$}
\psfrag{A}{\footnotesize$P_1$}
\psfrag{B}{\footnotesize$P_2$}
\psfrag{C}{\footnotesize$P_3$}
\psfrag{E}{\footnotesize$Q_1$}
\psfrag{F}{\footnotesize$Q_2$}
\psfrag{G}{\footnotesize$S$}
\psfrag{s}{\tiny$s_{1}$}
\psfrag{t}{\tiny$s_{2}$}
\psfrag{u}{\tiny$s_{3}$}
 \includegraphics[width=7cm]{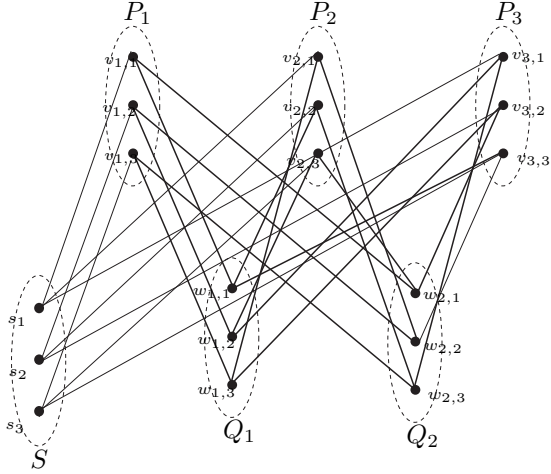}\\
\caption{Graph $H_n$ when $n=3$ in Step 1. The bold edges induce the graph obtained by removing the edges in $\bigcup_{j \in [3]} \{ xy : x,y\in T_j\}$ from  graph $G_3$ in \cite{KP2014-LSCC}. }\label{fig1}
\end{center}
\end{figure}

\bigskip
Given a graph $H$ and a vertex $v$ in $H$, {\em duplicating} $v$ means
adding a new vertex $v_0$ and making it adjacent to all the neighbors of
$v$ in $H$, but $v$ and $v_0$ are not adjacent. (See Figure~\ref{fig_d} for an illustration.)

\begin{figure}[b!]
\begin{center}
\psfrag{a}{\footnotesize$w_1$}
\psfrag{b}{\footnotesize$w_2$}
\psfrag{c}{\footnotesize$w_3$}
\psfrag{d}{\footnotesize$v$}
\psfrag{f}{\footnotesize$v_0$}
\includegraphics[width=7.5cm]{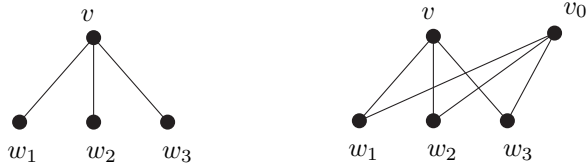}\\
\caption{The graph on the right is obtained by duplicating the vertex $v$ of the graph on the left.}\label{fig_d}
\end{center}
\end{figure}

\medskip
\noindent {\bf Step 2: } Duplicate each vertex of $\cup_{k=1}^{n}P_k$ exactly $(n-1)$ times.
For each vertex $v_{k, j}$, denote the $(n-1)$ copies of $v_{k, j}$ by
$v_{k, j}^2, v_{k, j}^3, \ldots, v_{k, j}^n$, and denote the original vertex $v_{k, j}$ by $v_{k, j}^1$.

\bigskip
Let  $T_{l,m}=\{ v^l_{1,m}, v^l_{2,m}, \ldots, v^l_{n,m}\}$ for each $m \in [n]$.

\medskip
\noindent {\bf Step 3: }
Introduce new $n^2-n$ vertices of $\cup_{i = 1}^{n-1} R_i$ where
$R_i = \{ u_{i,1}, u_{i,2}, ...., u_{i,n}\}$.
For each vertex $u_{i,j}$, make $u_{i, j}$ adjacent to all vertices in $\cup_{l= 1}^{n} T_{l,L_i(j,l)}$.
Note that the neighborhood of $u_{i, j}$ follows the same pattern of the neighborhood of $w_{i, j}$.  For example, if $N_{H_n} (w_{i, j}) = \{v_{k,L_i(j,k)} :  k \in [n] \}$, then $N_G(u_{i, j}) = \cup_{k=1}^n T_{k,L_i(j,k)}$.
Now, call the resulting graph by  $G$.

\bigskip

\begin{figure}
\begin{center}
\psfrag{A}{\footnotesize$v_{1,1}$}
\psfrag{B}{\footnotesize$v_{1,2}$}
\psfrag{C}{\footnotesize$v_{1,3}$}
\psfrag{x}{\footnotesize${}^{1}$}
\psfrag{y}{\footnotesize${}^{2}$}
\psfrag{z}{\footnotesize${}^{3}$}
\psfrag{D}{\footnotesize$v_{2,1}$}
\psfrag{E}{\footnotesize$v_{2,2}$}
\psfrag{F}{\footnotesize$v_{2,3}$}
\psfrag{G}{\footnotesize$v_{3,1}$}
\psfrag{H}{\footnotesize$v_{3,2}$}
\psfrag{I}{\footnotesize$v_{3,3}$}
\psfrag{J}{\footnotesize$w_{1,1}$}
\psfrag{K}{\footnotesize$w_{1,2}$}
\psfrag{L}{\footnotesize$w_{1,3}$}
\psfrag{S}{\footnotesize$w_{2,1}$}
\psfrag{T}{\footnotesize$w_{3,2}$}
\psfrag{U}{\footnotesize$w_{3,3}$}
\psfrag{P}{\footnotesize$u_{1,1}$}
\psfrag{Q}{\footnotesize$u_{1,2}$}
\psfrag{R}{\footnotesize$u_{1,3}$}
\psfrag{V}{\footnotesize$u_{2,1}$}
\psfrag{W}{\footnotesize$u_{2,2}$}
\psfrag{X}{\footnotesize$u_{2,3}$}
\psfrag{M}{\footnotesize$s_1$}
\psfrag{N}{\footnotesize$s_2$}
\psfrag{O}{\footnotesize$s_3$}
\psfrag{a}{\footnotesize$P_1^1$}
\psfrag{b}{\footnotesize$P_2^1$}
\psfrag{c}{\footnotesize$P_3^1$}
\psfrag{d}{\footnotesize$P_1^2$}
\psfrag{e}{\footnotesize$P_2^2$}
\psfrag{f}{\footnotesize$P_3^2$}
\psfrag{g}{\footnotesize$P_1^3$}
\psfrag{h}{\footnotesize$P_2^3$}
\psfrag{i}{\footnotesize$P_3^3$}
\psfrag{j}{\footnotesize$Q_1$}
\psfrag{k}{\footnotesize$Q_2$}
\psfrag{l}{\footnotesize$R_1$}
\psfrag{m}{\footnotesize$R_2$}
\psfrag{n}{\footnotesize$S$}
\psfrag{o}{\footnotesize$T_{1,1}$}
\psfrag{q}{\footnotesize$T_{1,2}$}
\psfrag{r}{\footnotesize$T_{1,3}$}
\psfrag{s}{\footnotesize$T_{2,1}$}
\psfrag{t}{\footnotesize$T_{2,2}$}
\psfrag{u}{\footnotesize$T_{2,3}$}
\psfrag{v}{\footnotesize$T_{3,1}$}
\psfrag{w}{\footnotesize$T_{3,2}$}
\psfrag{1}{\footnotesize$T_{3,3}$}
\psfrag{Z}{\footnotesize$\begin{array}{ll}
N_G(w_{1,1})=\cup_{l=1}^{3}\{ v_{1,\bf{1}}^l,v_{2,\bf{2}}^1,v_{3,\bf{3}}^1\},  &  N_G(w_{2,1})=\cup_{l=1}^{3}\{ v_{1,\bf{1}}^l,v_{2,\bf{3}}^l,v_{3,\bf{2}}^l\}  \\ & \\
N_G(w_{1,2})=\cup_{l=1}^{3}\{ v_{1,\bf{2}}^l,v_{2,\bf{3}}^l,v_{3,\bf{1}}^l\},  &  N_G(w_{2,2})=\cup_{l=1}^{3}\{ v_{1,\bf{2}}^l,v_{2,\bf{1}}^l,v_{3,\bf{3}}^l\}  \\ & \\
N_G(w_{1,3})=\cup_{l=1}^{3}\{ v_{1,\bf{3}}^l,v_{2,\bf{1}}^l,v_{3,\bf{2}}^l\},  &   N_G(w_{2,3})=\cup_{l=1}^{3}\{ v_{1,\bf{3}}^l,v_{2,\bf{2}}^l,v_{3,\bf{1}}^l\}
\end{array}$}
\psfrag{Y}{\footnotesize$\begin{array}{ll}
N_G(u_{1,1})= T_{1,\bf{1}} \cup T_{2,\bf{2}} \cup T_{3,\bf{3}},
& N_G(u_{2,1})= T_{1,\bf{1}} \cup T_{2,\bf{3}} \cup T_{3,\bf{2}} \\
& \\
N_G(u_{1,2})= T_{1,\bf{2}} \cup T_{2,\bf{3}} \cup T_{3,\bf{1}},
&  N_G(u_{2,2})= T_{1,\bf{2}} \cup T_{2,\bf{1}} \cup T_{3,\bf{3}} \\
 & \\
N_G(u_{1,3})= T_{1,\bf{3}}\cup T_{2,\bf{1}}\cup T_{3,\bf{2}},
&  N_G(u_{2,3})=  T_{1,\bf{3}}\cup T_{2,\bf{2}}\cup T_{2,\bf{1}}
\end{array}$}
\psfrag{p}{\footnotesize$\begin{array}{l}
N_G(s_1)=\cup_{l=1}^{3} T_{l,\bf{1}} ,
\quad N_G(s_2)=\cup_{l=1}^{3} T_{l,\bf{2}},
\quad  N_G(s_3)=\cup_{l=1}^{3} T_{l,\bf{3}}
\end{array}$}
\includegraphics[scale=0.85]{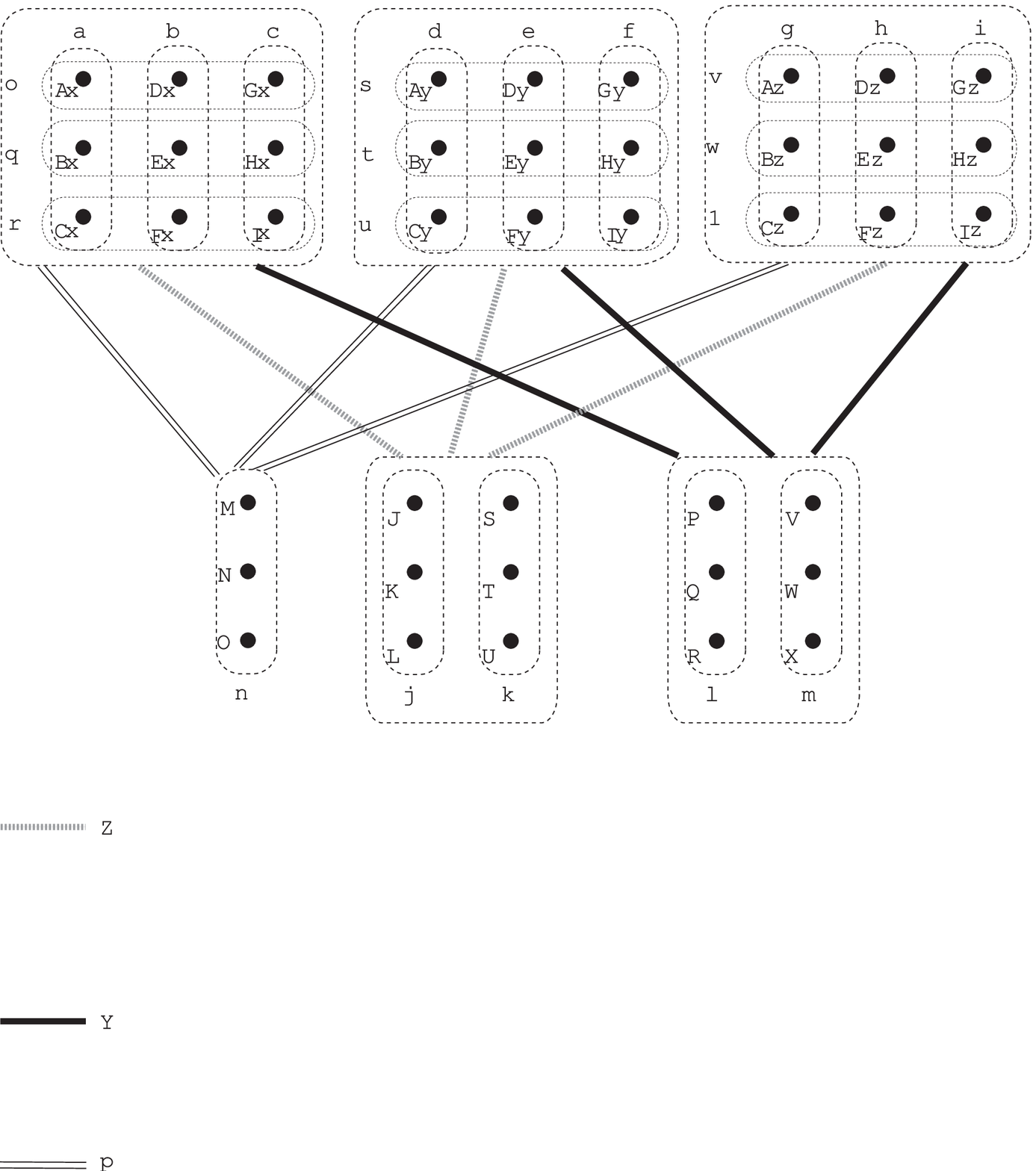}\\

\caption{ Graph $G$ when $n=3$.}\label{fig2}
\end{center}
\end{figure}

Figure~\ref{fig2} is an illustration of $G$ when $n=3$, and its description is below.

\bigskip

\noindent\textbf{Description of Figure~\ref{fig2}: }  
For each $l \in [3]$, the dotted line abbreviates  adjacency between $P_1^l\cup P_2^l\cup P_3^l$ and $Q_1\cup Q_2$. For each $l \in [3]$, the bold line abbreviates adjacency between $P_1^l\cup P_2^l\cup P_3^l$ and
$R_1\cup R_2$, and the doubled thin line abbreviates adjacency between the union of all $P_k^l$'s and $S$. In  $N_G(w_{i,j})$ and $N_G(u_{i, j})$, the bold subscripts are the $j$th row of the Latin square $L_i$ which was defined in Figure \ref{Latin-square}, respectively.
Note that for each $l \in [3]$ the subgrah induced by $P_1^l\cup P_2^l\cup P_3^l\cup Q_1\cup Q_2\cup S$
is isomorphic to graph $H_3$ in Figure~\ref{fig1}.

\bigskip

The following is a formal description of the construction of $G$.

\begin{construction}\label{construction} \rm
We construct a graph $G$ with $n (n^2+2n -1)$ vertices as follows.
For each $k, l \in [n]$, let $P_k^l$ be the set of $n$ elements such that
\begin{eqnarray*}
P_k^l&=&\{ v_{k,1}^l, v_{k,2}^l, ..., v_{k,n}^l \},
\end{eqnarray*}
and for each $i\in[n-1]$, let $Q_i$ be the set of $n$ elements such that
\begin{eqnarray*}
Q_i&=&\{ w_{i,1}, w_{i,2}, ...., w_{i,n}\},
\end{eqnarray*}
and for $i\in[n-1]$, let $R_i$ be the set of $n$ elements such that
\begin{eqnarray*}
R_i&=&\{ u_{i,1}, u_{i,2}, ...., u_{i,n}\},
\end{eqnarray*}
and let
\begin{eqnarray*}
S&=&\{ s_{1}, s_{2}, ...., s_{n}\}.
\end{eqnarray*}
Let $\{L_1,L_2,\ldots, L_{n-1}\}$ be the family of mutually orthogonal Latin squares of order $n$
obtained by (\ref{eq:Latin}).
For each $l, m \in[n]$,  let
\begin{eqnarray*} \label{T}
T_{l, m} & = &\{ v_{1, m}^l, v_{2,m}^l, \ldots, v_{n,m}^l\}.
\end{eqnarray*}
Now we define a graph $G$ as follows:
\begin{eqnarray*}
V(G) &=& \left( \cup_{l = 1}^n \cup_{k =1}^{n} P_k^l  \right) \bigcup \ \left( \cup_{i =1}^{n-1} Q_i \right) \bigcup \ \left( \cup_{i = 1}^{n-1} R_i \right) \cup S,\\
E(G) &=& E_1 \cup E_2 \cup \cdots \cup E_n \cup E_{n+1} \cup E_{n+2},
\end{eqnarray*}
where
\begin{eqnarray*} \label{E1-E2}
E_l&=&\bigcup_{i\in [n-1]}\bigcup_{j\in [n]} \{ w_{i,j}  v_{k,L_i(j,k)}^l : k\in [ n] \}, \mbox{ for each }  l \in[n], \\
E_{n+1}&=& \bigcup_{i \in [n-1]}\bigcup_{j\in [n]} \{ u_{i,j}  y  : y \in T_{l,L_i(j,l)} \mbox{ for some }  l \in[n] \}, \\
E_{n+2}&=& \bigcup_{m \in [n]} \{ s_{m}  y  : y \in T_{l, m} \mbox{ for some }  l \in[n]\}.
\end{eqnarray*}
\end{construction}

\bigskip

By the definition of the graph $G$, it follows that
\begin{eqnarray} \label{Neighbor(w)}
N_G(w_{i,j})&=&\bigcup_{l\in [n]} \{ v^l_{1,L_i(j,1)} ,v^l_{2,L_i(j,2)}, \ldots,  v^l_{n,L_i(j,n)} \}, \\
N_G(u_{i,j})&=& T_{1,L_i(j,1)} \cup T_{2,L_i(j,2)} \cup \cdots  \cup T_{n,L_i(j,n)},  \label{Neighbor(u)}\\
N_G(s_m)&=& T_{1, m}\cup T_{2,m} \cup \cdots \cup T_{n,m}.  \label{Neighbor(s)}
\end{eqnarray}
For simplicity, for each $l\in [n]$, let $P^l=P_1^l \cup  \cdots \cup P_n^l$ and let
\begin{eqnarray*}
P &=& P^1 \cup  \cdots \cup P^{n},   
\\
Q&=&Q_1\cup  \cdots \cup Q_{n-1},\\
R&=&R_1\cup  \cdots \cup R_{n-1}.
\end{eqnarray*}
Let $K_{n*r}$ denote the complete multipartite graph with $r$ partite sets in which each partite set has $n$ vertices.
We will show that the subgraph of $G^2$ induced by $P$ is the complete multipartite graph $K_{n \star n^2}$ whose partite sets are $\{P_k^l : k,l\in [n] \}$.

\medskip
For each $l\in [n]$, let $G_l$ be the subgraph of $G$ induced by $P^l \cup Q$. 
The following properties were obtained in Lemma 2.2 and Lemma 2.4 in \cite{KP2014-LSCC}.

\begin{lemma}\label{N(w)} (\cite{KP2014-LSCC})
For each $l\in [n]$, $G_l$ satisfies the following properties.
\begin{itemize}
\item[\rm(1)] For any vertex $w\in Q$,
\[|N_{G_l} (w) \cap P_k^l|=1, \mbox{ for each }  k\in [n].\]
\item[\rm(2)] For any distinct vertices $w$ and $w'$ in $Q$,
\[|N_{G_l}(w)\cap N_{G_l}(w')|\le 1.\]
\item[\rm(3)] For any vertex $w\in Q$,
\[|N_{G_l}(w)\cap T_{l,m}|=1, \mbox{ for each }  m\in[n].\]
\item[\rm (4)] For any vertex $v\in P^l$,
\[|N_{G_l}(v) \cap Q_i| = 1, \mbox{ for each }  i\in [n-1].\]
\end{itemize}
\end{lemma}

From Lemma~\ref{N(w)}, we show the following lemmas.

\begin{lemma}\label{independent}
\begin{itemize}
\item[\rm (1)] For each $k, l \in [n]$, $P_k^l$ is an independent set of $G^2$.
\item[\rm (2)] For each $i \in [n - 1]$, $Q_i$ and $R_i$ are independent sets of $G^2$.
\item[\rm (3)] The set $S$ is an independent set of $G^2$.
\end{itemize}
\end{lemma}

\begin{proof}
Consider $P_k^l$ for some  $k, l \in [n]$.
Let $v$, $v'$ be distinct vertices in $P_k^l$. We will show that $v$ and $v'$ do not have a common neighbor.
First, the vertices $v$ and $v'$ do not have a common neighbor in $Q$ by (1) of Lemma~\ref{N(w)}.
Next, $|N_G(u)\cap P_k^l|= 1$ for any $u\in R$ by (\ref{Neighbor(u)}), and
$|N_G(s)\cap P_k^l|= 1$ for any $s\in S$ by (\ref{Neighbor(s)}).  Hence $v$ and $v'$ do not have a common neighbor in $R \cup S$.  Thus $v$ and $v'$ do not have a common neighbor in $G$, and  consecutively,
$v$ and $v'$ are not adjacent in
$G^2$. Therefore, $P_k^l$ is an independent set in $G^2$.

Let $w$ and $w'$ be any distinct vertices in $Q_i$. Suppose that the vertices $w$ and $w'$ are adjacent in $G^2$.
Since $G$ is a bipartite graph, they have a common neighbor $v$ in $P$. Then $v\in P^l$ for some $l$.
Thus $w,w'\in N_{G_l}(v)\cap Q_i$.  But,
by (4) of Lemma~\ref{N(w)}, $|N_{G_l}(v)\cap Q_i|=1$.  This is a contradiction for the assumption that $w$ and $w'$ are distinct.
Therefore for each $i \in [n-1]$, $Q_i$ is an independent set in $G^2$.

Let $u=u_{i,j}$ and $u'=u_{i,j'}$ be any distinct vertices in $R_i$. 
Suppose that the vertices $u$ and $u'$ are adjacent in $G^2$.
Then they have a common neighbor $v$ in $P$.
Then $v \in N_G(u)\cap N_G(u')$, and so  by (\ref{Neighbor(u)}),  $v\in T_{a,L_i(j,a)} \cap T_{b,L_i(j',b)}$ for some $a$ and $b$.
Therefore, $a=b$ and $L_i{(j,a)}=L_i{(j',a)}$, which implies $j=j'$ since $L_i$ is a Latin square.  This is a contradiction.
Thus for each $i \in [n-1]$, $R_i$ is an independent set in $G^2$.

Moreover, it is clear that by (\ref{Neighbor(s)}), any two vertices in $S$ do not have a common neighbor in $G$, and so $S$ is an independent set of $G^2$.
\end{proof}

Let $G^2[P^l]$ denote the subgraph of $G^2$ induced by $P^l$.

\begin{lemma} \label{step1}
For each $l\in [n]$, $G^2[P^l]  \cong  K_{n *n}$ whose partite sets are $P_{1}^l$, $P_2^l$,\ldots, $P_n^l$.
\end{lemma}
\begin{proof}
The proof of Lemma~\ref{step1} is basically similar to the proof of Lemma 2.8 in \cite{KP2014-LSCC}, but we will include here for the sake of completeness.
Take an integer  $l\in [n]$.
Note that $N_{G_l}(w) \subset P^l$ for each $w \in Q$ by the definition of $G_l$.  First, note that $G^2[P^l]$ is isomorphic to a subgraph of $K_{n \star n}$,
since  for each $k, l \in [n]$, $P_k^l$ is an independent set of $G^2[P^l]$ by Lemma \ref{independent}.
Let
\[\mathcal{F}_l = \{ G^2[N_{G}(w)\cap P^l] : w \in Q\} \cup \{ G^2 [N_{G}(s) \cap P^l] : s\in S \}.\]
Note that for each $w \in Q$, $N_{G}(w)\cap P^l$ induces a complete graph $K_n$ in $G^2$, and for each $s\in S$,  $N_{G}(s)\cap P^l$ induces a complete graph $K_n$ in $G^2$.
Therefore $\mathcal{F}_l$ is a family of  copies of $K_n$.

For any two vertices $w$, $w' \in Q$, we have $|N_{G_l}(w) \cap N_{G_l}(w')| \leq 1$  by (2) of Lemma~\ref{N(w)} and so
$|N_{G}(w)\cap N_{G}(w')\cap P^l| \le 1$. This implies that  $G^2[N_{G}(w)\cap P^l]$ and $G^2[N_{G}(w')\cap P^l]$ are edge-disjoint.
Note that $N_{G}(s_m)\cap P^l=T_{l,m}$ for each $m\in [n]$.
Thus $T_{l,m}\cap T_{l,m'}=\emptyset$ if $m\neq m'$. This implies that if $s\neq s'$, then $G^2[N_{G}(s)\cap P^l]$ and $G^2[N_{G}(s')\cap P^l]$ are edge-disjoint.
Next, by (3) of Lemma~\ref{N(w)}, for each $m\in [n]$, $|N_{G_l}(w) \cap T_{l,m}|=1$.
Thus $|N_{G}(w) \cap P^l \cap T_{l,m}|= 1$.  This implies that  $G^2[N_{G}(w)\cap P^l]$ and $G^2[N_{G}(s)\cap P^l]$ are edge-disjoint.
Therefore any two cliques in $\mathcal{F}_l$ are edge-disjoint.

In addition, $|\mathcal{F}_l| = |Q|+|S| = n(n - 1)+n=n^2$. 
Thus $\mathcal{F}_l$ is a family of $n^2$ pairwise edge-disjoint cliques of size $n$ in $G^2[P^l]$.
It follows that
\[|E(K_{n*n})| \geq |E(G^2[P^l])| \geq n^2  \times {n \choose 2} = |E(K_{n*n})|.\]
Hence $G^2[P^l]  \cong  K_{n *n}$ for each $l\in [n]$, since $G^2[P^l]$ is isomorphic to a subgraph of $K_{n *n}$.
\end{proof}

To show that $G^2[P] \cong K_{n \star n^2}$, it is remained to show the following lemma.

\begin{lemma}\label{st:adjacent}
For any distinct $s, t \in [n]$,  for any $v \in P^s$ and $v' \in P^t$, $v$ and  $v'$ are adjacent in $G^2$.
\end{lemma}

\begin{proof}
Let $v \in P^s$ and $v' \in P^t$ be any vertices with $s \neq t$.  Then $v\in T_{s, a}$ and $v'\in T_{t, b}$ for some $a, b \in [n]$.
If $a = b$, then $T_{s,a}\cup T_{t,b}\subset N_G(s_a)$ and so $v$ and $v'$ have a common neighbor $s_a$ in $G$.
Hence $vv' \in E(G^2)$.

We will show that if $a \neq b$, then there exist $i \in [n-1]$ and $j \in [n]$ such that  $L_i (j, s)=a$  and $L_i (j, t)=b$ for fixed $s$ and $t$.  Note that if $a \neq b$ and $s \neq t$, then there exist $i$ and $j$ satisfying the following equations.
\begin{eqnarray*}
 & j + i(s-1)  \equiv a &\pmod{n} \\
& j + i(t-1)  \equiv b  &\pmod{n}.
\end{eqnarray*}
Thus from (\ref{eq:Latin}), we know that there exist $i \in [n-1]$ with  $L_i (j, s)=a$ since $a \neq b$, 
and $j \in [n]$ with $L_i (j, t)=b$.
Note that by (\ref{Neighbor(u)}), we have $ T_{s, L_i(j,s)} \cup  T_{t,L_i(j,t)} \subset N_G(u_{i,j})$. Therefore   $ T_{s,a} \cup T_{t, b} \subseteq N_G(u_{i, j})$, and so
$v$ and $v'$ have a common neighbor $u_{i,j}$ in $G$.  Hence $vv'\in E(G^2)$.
\end{proof}


By Lemmas~\ref{independent},~\ref{step1} and~\ref{st:adjacent}, the following theorem holds.

\begin{theorem}\label{main-bipartite}
If $G$ be the graph defined in Construction \ref{construction}, then
$G^2[P] \cong K_{n \star n^2}$ whose partite sets are $P_{k}^l$'s.
\end{theorem}

The following lower bound on the list chromatic number of a complete multipartite graph was obtained in \cite{Vetrik2012}.

\begin{theorem}\label{thm:Vetrik}{\rm (Theorem 4, \cite{Vetrik2012})}
For a complete multipartite graph $K_{n*r}$ with $n,r \ge 2$,
\[\chi_\ell (K_{n*r}) > (n-1)\left\lfloor\frac{2r-1}{n} \right\rfloor.\]
\end{theorem}

Consequently, we obtain that  $\chi_{\ell}(G) >\chi(G)$ by the following theorem.

\begin{theorem} \label{main-theorem}
For each prime $n \geq 3$, if $G$ is the graph defined in Construction~\ref{construction}, then
\[ \chi_{\ell}(G^2) - \chi(G^2)  >  n^2 - 6n + 3. \]
\end{theorem}
\begin{proof}
It is clear that $\chi(G^2) \leq n^2 + 2n - 1 $  by Lemma~\ref{independent}.
On the other hand, by Theorems~\ref{main-bipartite} and~\ref{thm:Vetrik},
\begin{eqnarray*}
\chi_\ell(G^2) \geq \chi_{\ell} (K_{n\star n^2}) > (n-1)\left\lfloor\frac{2 n^2 -1}{n} \right\rfloor \geq 2(n-1)^2.
\end{eqnarray*}
Thus
\[ \chi_\ell(G^2) - \chi(G^2) > 2(n-1)^2 - (n^2 + 2n -1) = n^2 - 6n + 3.\]
\end{proof}

\begin{remark} \rm
Note that for any prime $n \geq 7$, we have $
 \chi_{\ell}(G^2) - \chi(G^2)  > n^2 - 6n + 3 > 0$.
Thus from Theorem~\ref{main-theorem}, there exists a bipartite graph $G$ such that $G^2$ is not chromatic-choosable.  Furthermore, since
there are infinitely many primes,
the gap $\chi_{\ell}(G^2) - \chi(G^2)$ can be arbitrarily large.
\end{remark}


\section{Further Discussion}

Note that from (4) of Lemma \ref{N(w)}, each vertex $v$ in $P$ has exactly one neighbor in each of $Q_i$, $R_j$, and $S$, respectively.  Thus, if $G$ is the bipartite graph  defined in Construction \ref{construction} for prime number $n$, then
$d_G(x) = 2n -1$ for each $x \in P$ and $d_G(y) = n^2$ for each $y \in Q \cup R \cup S$.  Hence from Theorem \ref{main-theorem}, if $G$ is the bipartite graph defined in Construction \ref{construction} for $n = 7$, then $G^2$ is not chromatic-choosable and every vertex of one partite set of $G$ has degree  13.
Note that the List Total Coloring Conjecture is true if the List Square Coloring Conjecture holds for bipartite graphs such that every vertex of one partite set has degree at most 2.  Thus, it would be interesting to answer the following questions.

\begin{question} \label{question-bipartite}
If $G$ is a bipartite graph such that every vertex of one partite set has degree at most 2, then is it true that  $\chi_{\ell}(G^2)=\chi(G^2)$?
\end{question}

\begin{question} \label{question-bipartite2}
If Question \ref{question-bipartite} is true, then what is the largest $k$ such that $G^2$ is chromatic-choosable for every bipartite graph $G$ with a partite set in which each vertex has degree at most $k$?
\end{question}

We already mentioned that there is a bipartite graph $G$ such that every vertex of one partite set of $G$ has degree $13$ and $G^2$ is not chromatic-choosable. Thus if Question \ref{question-bipartite} is true (or the List Total Coloring Conjecture is true), then
the $k$ in Question~\ref{question-bipartite2} must be less than 13.

On the other hand,
if we apply to the `duplication idea' in Step 2 in the procedure  of the construction of $G$ repeatedly,
then we can obtain a bipartite graph $\mathcal{G}$ such that  every vertex of one partite set of $\mathcal{G}$ has degree  $7$  and $\mathcal{G}^2$ is not chromatic-choosable.  This implies that the integer $k$ in
Question~\ref{question-bipartite2} must be less than 7.


We will describe briefly how to construct such bipartite graph $\mathcal{G}$.
Let
$G$ be the graph in Construction \ref{construction} when $n=3$.
Now we duplicate each vertex of $P$ exactly $2$ times.
For each  vertex $v_{k,j}^l$, we denote its copies by ${v'}_{k, j}^{l}$ and  ${v''}_{k, j}^{l}$.
Let $P'$ denote the set of the first copied vertices ${v'}_{k, j}^{l}$, and let $P''$ denote the set of the second copied vertices ${v''}_{k, j}^{l}$.
For each $h\in[3]$, let $\mathcal{T}_{1,h}=T_{1,h}\cup T_{2,h} \cup T_{3,h} $, that is,
\[ \mathcal{T}_{1,h}= \{ v_{1,h}^{1}, {v}_{2,h}^{1}, {v}_{3,h}^{1}, {v}_{1,h}^{2}, {v}_{2,h}^{2}, {v}_{3,h}^{2}, {v}_{1,h}^{3}, 
 {v}_{2,h}^{3},  {v}_{3,h}^{3}\}.\]
In addition, let the two copies corresponding to $\mathcal{T}_{1,h}$ be denoted as follows:
\begin{eqnarray*}
&&\mathcal{T}_{2,h}= \{ {v'}_{1,h}^{1}, {v'}_{2,h}^{1}, {v'}_{3,h}^{1}, {v'}_{1,h}^{2}, {v'}_{2,h}^{2}, {v'}_{3,h}^{2}, {v'}_{1,h}^{3},
 {v'}_{2,h}^{3},  {v'}_{3,h}^{3} \} ,\\
&&\mathcal{T}_{3,h}=  \{ {v''}_{1,h}^{1}, {v''}_{2,h}^{1}, {v''}_{3,h}^{1}, {v''}_{1,h}^{2}, {v''}_{2,h}^{2}, {v''}_{3,h}^{2}, {v''}_{1,h}^{3},
 {v''}_{2,h}^{3},  {v''}_{3,h}^{3} \}.
 \end{eqnarray*}
Next, we introduce $6$ new vertices of $B_1\cup B_2$ where $ B_1=\{b_{1,1},b_{1,2},b_{1,3}\}$ and
$B_2=\{b_{2,1},b_{2,2},b_{2,3}\}$, in which the neighborhood of each vertex $b_{i,j} \in B_1 \cup B_2$ follows the same pattern of the neighborhood of $w_{i,j}$ (similar to Step 3 in the procedure  of the construction of $G$).
More precisely, $N_{\mathcal{G}}(b_{i,j})=\cup_{k=1}^{3} \mathcal{T}_{k,L_i(j,k)}$ for each $b_{i,j}$, where  $\mathcal{G}$ is the resulting graph.
See Figure~\ref{fig3} for an illustration and its description is below.

\begin{figure}[h!]
\begin{center}
\psfrag{O}{\tiny$B_1$}
\psfrag{P}{\tiny$B_2$}
\psfrag{j}{\tiny$Q_1$}
\psfrag{k}{\tiny$Q_2$}
\psfrag{l}{\tiny$R_1$}
\psfrag{m}{\tiny$R_2$}
\psfrag{n}{\tiny$S$}
\psfrag{o}{\tiny$\mathcal{T}_{1,1}$}
\psfrag{q}{\tiny$\mathcal{T}_{1,2}$}
\psfrag{r}{\tiny$\mathcal{T}_{1,3}$}
\psfrag{d}{\tiny$\mathcal{T}_{2,1}$}
\psfrag{e}{\tiny$\mathcal{T}_{2,2}$}
\psfrag{f}{\tiny$\mathcal{T}_{2,3}$}
\psfrag{g}{\tiny$\mathcal{T}_{3,1}$}
\psfrag{h}{\tiny$\mathcal{T}_{3,2}$}
\psfrag{i}{\tiny$\mathcal{T}_{3,3}$}
\psfrag{J}{\tiny$b_{1,1}$}
\psfrag{K}{\tiny$b_{1,2}$}
\psfrag{L}{\tiny$b_{1,3}$}
\psfrag{S}{\tiny$b_{2,1}$}
\psfrag{T}{\tiny$b_{2,2}$}
\psfrag{U}{\tiny$b_{2,3}$}
\psfrag{H}{\scriptsize$P$}
\psfrag{E}{\scriptsize$P'$}
\psfrag{C}{\scriptsize$P''$}
\psfrag{B}{\scriptsize$S\cup Q \cup R$}
\psfrag{V}{\scriptsize Graph ${G}$}
\psfrag{z}{\footnotesize$\begin{array}{ll}
N_{\mathcal{G}}(b_{1,1})= \mathcal{T}_{1,\bf{1}} \cup \mathcal{T}_{2,\bf{2}} \cup \mathcal{T}_{3,\bf{3}},
& \quad N_{\mathcal{G}}(b_{2,1})= \mathcal{T}_{1,\bf{1}} \cup \mathcal{T}_{2,\bf{3}} \cup \mathcal{T}_{3,\bf{2}} \\
& \\
N_{\mathcal{G}}(b_{1,2})= \mathcal{T}_{1,\bf{2}} \cup \mathcal{T}_{2,\bf{3}} \cup \mathcal{T}_{3,\bf{1}},
& \quad N_{\mathcal{G}}(b_{2,2})= \mathcal{T}_{1,\bf{2}} \cup \mathcal{T}_{2,\bf{1}} \cup \mathcal{T}_{3,\bf{3}} \\
 & \\
N_{\mathcal{G}}(b_{1,3})= \mathcal{T}_{1,\bf{3}}\cup \mathcal{T}_{2,\bf{1}}\cup \mathcal{T}_{3,\bf{2}},
& \quad N_{\mathcal{G}}(b_{2,3})=  \mathcal{T}_{1,\bf{3}}\cup \mathcal{T}_{2,\bf{2}}\cup \mathcal{T}_{2,\bf{1}}
\end{array}$}
\includegraphics[scale=0.75]{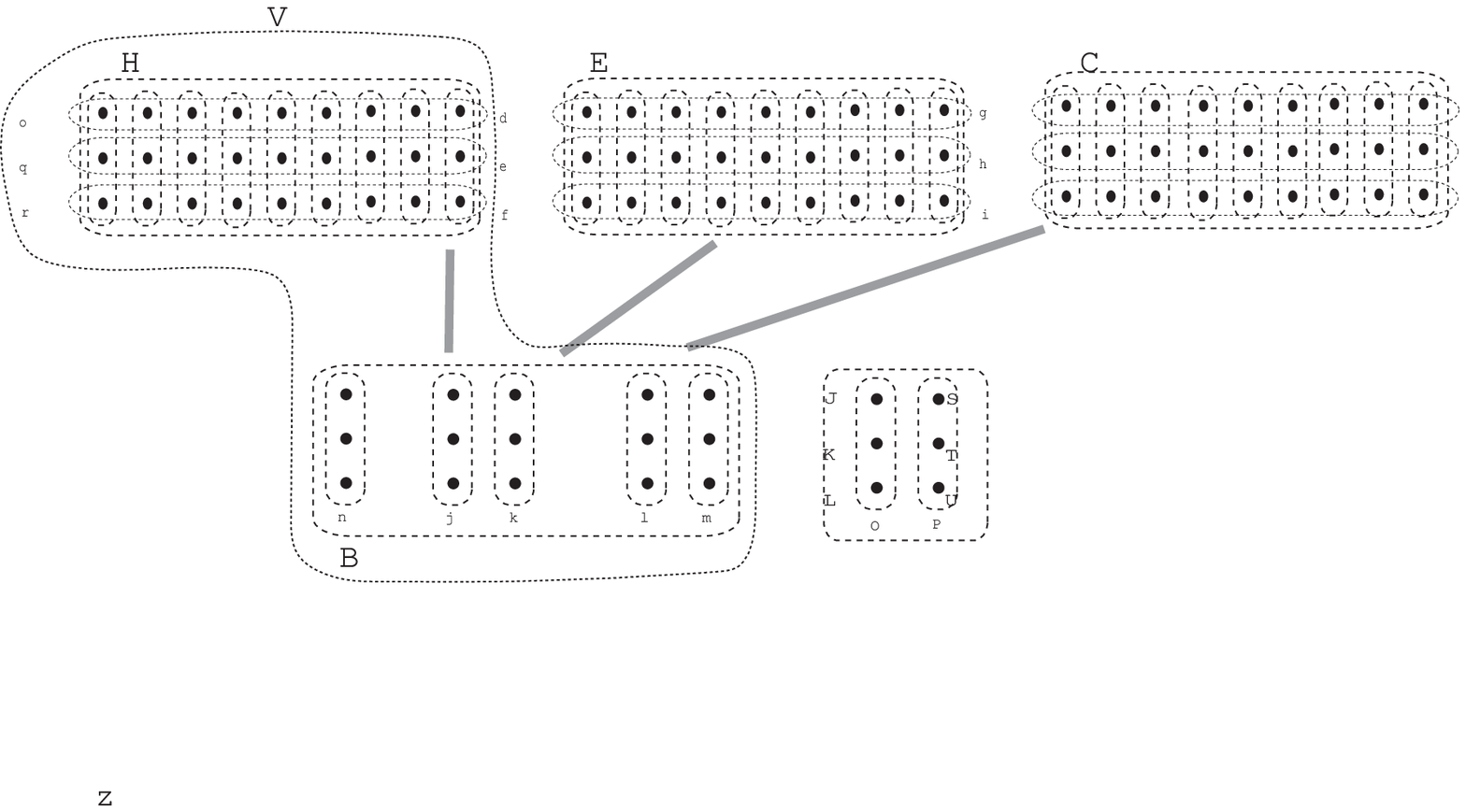}\\
\caption{Graph $\mathcal{G}$}\label{fig3}
\end{center}
\end{figure}
\bigskip

\noindent\textbf{Description of Figure~\ref{fig3}: }  
The sets $P'$ and $P''$ are copies of $P$, and the bold line
abbreviates  adjacency between $P$ and $S\cup Q \cup R$. Each of three $P\cup S\cup Q\cup R$, $P'\cup S\cup Q\cup R$, $P''\cup S\cup Q\cup R$ induces a graph isomorphic to graph $G$ in Figure~\ref{fig2}.
Like as $N_G(w_{i,j})$ and $N_G(u_{i, j})$, in $N_{\mathcal{G}}(b_{i,j})$, the bold subscripts are the $j$th row of the Latin square $L_i$ which was defined in Figure \ref{Latin-square}.

\bigskip

Then, the resulting graph $\mathcal{G}$ is a bipartite graph with partite sets
$X=P\cup P'\cup P''$ and  $Y=S\cup Q \cup R \cup B_1 \cup B_2$.  
Note that each vertex $x \in X$ has degree 7 and each vertex $y \in Y$ has degree 27. 
Then for each $k, l \in [3]$, we can see that $P_k^l=\{v_{k,1}^l, v_{k,2}^l,v_{k,3}^l, \}$ is an independent set in $\mathcal{G}^2$ 
and each of its corresponding copies  ${P'}_k^l$ and  ${P''}_k^l$ is also an independent set in $\mathcal{G}^2$.
In addition, each of $S$, $Q_1$, $Q_2$, $R_1$, $R_2$, $B_1$, $B_2$ is an independent set in $\mathcal{G}^2$.
Thus we know that $\mathcal{G}^2$  is a multipartite graph with $34$ partite sets. Therefore $\chi(\mathcal{G}^2) \leq 34$.
Moreover, we can easily check that the subgraph of $\mathcal{G}^2$ induced by $X$ is the complete multipartite graph $K_{3\star 27}$, and so 
$\chi_\ell(\mathcal{G}^2) \ge \chi_\ell(K_{3\star 27}) =\lceil \frac{4\times 27 -1 }{3}\rceil =36$ (see~\cite{Kier}).
Thus ${\mathcal{G}}^2$ is not chromatic-choosable.

\begin{remark}
In general, for each prime number $n$, if we apply this duplication idea $d$ times to the graph $H_n$ in the Construction \ref{construction},
then we have a bipartite graph whose square is a multiparite graph with $n^{d} + d(n-1)+1$ partite sets, containing a complete multiparitite graph $K_{n \star n^{d}}$.  Through this way, we can also construct many bipartite graphs whose square are not chromatic-choosable.
\end{remark}

\bigskip

\noindent {\bf Acknowledgement.}
The first author (S.-J. Kim) was  supported by the National Research Foundation of Korea(NRF) grant funded by the Korea government (MEST)
 (No. 2011-0009729), and the second author (B. Park)
was supported by the National Institute for Mathematical Sciences (NIMS) grant funded by the Korea government (No. B21403-2).


\end{document}